%% file: manuscript_arxiv.tex
\documentclass[preprint,11pt,authoryear]{elsarticle}
\usepackage[T1]{fontenc}
\usepackage{lmodern}
\usepackage{amsmath,amssymb,amsthm}
\usepackage{graphicx}
\usepackage{algorithm}
\usepackage{algpseudocode}
\usepackage[hidelinks]{hyperref}
\usepackage{booktabs}
\usepackage{placeins}

\newtheorem{lemma}{Lemma}
\newtheorem{proposition}[lemma]{Proposition}
\newtheorem{theorem}[lemma]{Theorem}
\newtheorem{corollary}{Corollary}
\theoremstyle{remark}
\newtheorem{remark}{Remark}
\theoremstyle{definition}

\newcommand{\E}{\mathbb{E}}
\newcommand{\R}{\mathbb{R}}
\newcommand{\had}[1]{^{\circ #1}}

\journal{European Journal of Operational Research}

\begin{document}

\begin{frontmatter}

\title{Feasible closed form moment matching for scenario generation through covariance rotations}

\author[1]{Juan Pablo Contreras \corref{mycorrespondingauthor}}
\cortext[mycorrespondingauthor]{Corresponding author}
\ead{juan.contreras@udp.cl}
\address[1]{Escuela de Ingenier\'ia Industrial, Universidad Diego Portales, Ej\'ercito 441, Santiago, Chile}

\begin{abstract}
Closed form moment matching constructs small scenario sets without numerical
optimisation. We study two generators that match the mean, the covariance
matrix, and the marginal third and fourth moments. Both place their
scenarios along the columns of a covariance factor, usually the symmetric or
Cholesky root. Although every valid factor reproduces the same covariance
matrix, its choice can determine whether the resulting weights are
nonnegative. By viewing each column as a univariate mixture component,
Pearson's inequality gives the minimum probability needed to reproduce its
assigned variance, skewness, and kurtosis, which leads to exact feasibility
conditions whenever the relevant entrywise powers are nonsingular. The
effect of the factor is governed by its entrywise fourth powers, whose sum
can vary by a factor equal to the dimension. This explains why standard
roots fail in high dimension. We combine these results into a rotated moment
matching procedure that searches for feasible factors without altering the
target moments. Applied to stock returns and synthetic data, the procedure
produces valid probabilities for both generators in nearly every instance
considered, and it outperforms the original constructions, whose standard
factors give invalid weights in most of them. Finally, the two asset case
explains when such a rotation exists. With symmetric marginals, a suitable
rotation reaches Pearson's bound except in one degenerate case, whereas with
skewness the factor must balance the allocation of third and fourth moments,
and some attainable targets admit no feasible rotation. We also establish a
sufficient condition for arbitrary correlation and equal marginal kurtosis.
\end{abstract}


\end{frontmatter}

\section{Introduction}\label{sec:intro}

Stochastic programs replace an uncertain distribution with a finite set of
scenarios whose statistical quality can materially affect decisions in
portfolio selection, asset liability management, capacity planning, and
many other applications. Their computational cost is equally important
because a rolling study requires a new scenario set at every decision date,
while a multistage model may require one at every node.

Moment matching constructs scenarios that reproduce a small collection of
statistics regarded as reliable, usually the mean, covariance, marginal
skewness, and marginal kurtosis \citep{hoylandwallace2001}. Most methods
obtain such scenarios by solving an optimisation problem, whereas closed
form generators adapt sigma point constructions from nonlinear filtering and
express every point and weight directly in terms of the targets. Among the
latter, \citet{prd2015}, hereafter PRD, match the mean, covariance, and the
averages of the marginal third and fourth moments, while the generalised
unscented transform of \citet{genut2025}, hereafter GenUT, matches every
marginal third and fourth moment with \(2n+1\) points for \(n\) assets, so
both methods are exact and grow only linearly with the dimension.

The main practical difficulty is that some resulting weights can be
negative. \citet{cbh2018} address this for PRD by deriving parameter choices
that prevent the problem whenever such choices exist. Higher order filtering
methods sometimes raise a kurtosis target until their formulas become real
\citep{house2021,yang2024}, and GenUT implements the same adjustment in its
code, although changing the target also changes the distribution being
represented and does not control the centre weight. Negative weights may
remain usable for numerical integration or filtering, but they cannot
define scenarios for a stochastic program.

This paper identifies the covariance factor as a central part of the
feasibility question because both generators place points along the
columns of a matrix \(L\) satisfying \(LL^\top=\Sigma\). Since every such
factor can be written as \(L=\Sigma^{1/2}Q\) for an orthogonal matrix \(Q\),
a rotation preserves the mean and covariance
while changing how the third and fourth moments are allocated among columns.
Only the choice of factor requires a numerical optimisation over the
orthogonal group, and once the factor is fixed the scenarios and their
probabilities remain explicit functions of the targets.

As an example, consider two assets with correlation \(0.4\), marginal skewness \(0.3\), and
raw marginal kurtosis \(1.6\). These marginal moments satisfy Pearson's
inequality, but this alone does not ensure that GenUT produces nonnegative
weights. Figure~\ref{fig:geometry} compares two covariance factors for the
same target moments. Each colour identifies the pair of points generated by
one column of the factor, marker area represents the corresponding weight,
and the ellipse shows the two standard deviation contour. With the symmetric
factor in panel (a), the noncentral points require a total weight of \(1.22\),
which forces the remaining weight at the mean to be \(-0.22\). The resulting
scenario set therefore matches the target moments but does not define a
probability distribution. Rotating the factor toward the principal
directions changes how the third and fourth moments are allocated between
the two column components. In panel (b), the noncentral points require only
\(0.68\) of the total probability, leaving a positive centre weight of
\(0.32\). The rotation thus preserves the covariance, the marginal moments,
and the number of scenarios while restoring probabilistic feasibility.

\begin{figure}[t]
\centering
\includegraphics[scale=0.8]{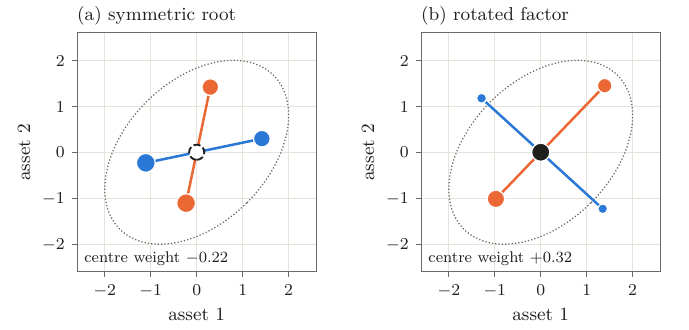}
\caption{Effect of the covariance factor on GenUT feasibility for the two
asset targets \(r=0.4\), \(S=0.3\), and \(K=1.6\), with the symmetric root in
(a) and a factor rotated close to the principal directions in (b). Marker
areas are proportional to the weights, and a dashed open marker shows a
negative centre weight.}
\label{fig:geometry}
\end{figure}

The analysis rests on the observation that each column is a one
dimensional mixture component, and Pearson's inequality, applied to such a
component together with its probability, gives the minimum mass required
for prescribed second, third, and fourth moment contributions.

We use this observation to obtain four main results. First, we establish a
general feasibility principle for scenario generators assembled from
univariate mixture components, which determines the minimum probability
required by each component from its assigned second, third, and fourth
moments. This principle yields exact conditions for column based generators
with nonsingular entrywise powers, recovers the optimal PRD probabilities,
and identifies the probability constraint missing from the published GenUT
check.

Second, we show that factor concentration is governed by the sum of the
fourth powers of its entries, which can vary by a factor equal to the
dimension, so that diffuse factors reduce the resulting dimension penalty,
and we combine these results into a rotated moment matching method that
searches over factors before any target is modified. Third, numerical
experiments using equity returns and synthetic data confirm the effect of
factor concentration and show that the method makes both generators valid on
every dataset tested except one on which no factor of the PRD form can
succeed. Fourth, for two assets with symmetric marginals we show that
rotation reaches Pearson's bound except in one limiting case, whereas with
skewness rotation creates a tradeoff between third and fourth moments, for
which we give a sufficient condition covering arbitrary correlation and equal
marginal kurtosis, and we illustrate these thresholds numerically.

The rest of the paper is organised as follows. Section~\ref{sec:literature}
reviews the related literature. Section~\ref{sec:setting} presents the two
constructions and their exact feasibility conditions, studies in
Section~\ref{sec:size} how the covariance factor changes those conditions,
and closes with the rotated moment matching method.
Section~\ref{sec:numerics} evaluates the method on real and synthetic data,
while Section~\ref{sec:two} establishes the two asset existence results that
explain when rotation succeeds. Section~\ref{sec:conclusion} concludes, and
the proofs not given in the main text are collected in \ref{app:proofs}.

\section{Related literature}\label{sec:literature}

Scenario generation methods for stochastic programming range from sampling
and quasi Monte Carlo schemes to moment matching, approximations based on
probability metrics, and problem driven constructions, and broad treatments
are given by \citet{kautwallace2007}, \citet{king2024}, and \citet{chou2023}.
The empirical comparison of \citet{lohndorf2016} shows that the relative
performance of these families depends on the problem and on the number of
scenarios, while importance sampling can reduce the number of scenarios
required in intertemporal portfolio choice by a factor of several hundred
\citep{ekblom2020}. Generators are usually assessed through the in sample and
out of sample stability of the optimal value \citep{kautwallace2007} or
through the quality of the resulting decisions \citep{keutchayan2017}. These
evaluation frameworks apply to rotated scenario sets as to any other, and
because different rotations preserve the matched moments while changing the
unmatched comoments and therefore possibly the decisions, we leave decision
quality aside and concentrate on the prior question of whether a valid set
exists.

Optimisation based moment matching was introduced by
\citet{hoylandwallace2001}, who minimise a weighted discrepancy between
scenario and target moments, and was made scalable by the heuristic of
\citet{hkw2003}, which alternates marginal transformations with corrections
of the correlation structure. \citet{mehrotrapapp2013} match moments exactly
through column generation, \citet{staino2015} solve the moment equations of
arbitrage free scenario trees with the monomial method, and
\citet{tanaka2015} discretise distributions with exact moments by minimising
relative entropy, which yields positive weights by construction. More recent
work selects scenarios from historical data through moment based
optimisation \citep{kaut2021}, formulates distribution and moment matching
as a mixed integer linear program \citep{bounitsis2022}, and reduces scenario
sets while preserving higher order moments in portfolio applications
\citep{zhang2022}, and moment matching remains a standard choice in energy
system models \citep{backe2021}. These methods obtain valid probabilities
through numerical optimisation whose cost recurs at every decision date or
tree node, while quadrature based alternatives such as sparse grids
\citep{chen2015} can themselves produce negative weights.

Closed form constructions originate in the unscented transform and its
higher order extensions
\citep{julier1998,tenne2003,julier2004,liu2014,menegaz2015,adurthi2018,easleyberry2021},
which place a small number of weighted points so as to reproduce prescribed
moments. \citet{datemamon2008} used such constructions to sample from
partially specified symmetric distributions, \citet{ponomarevadate2013}
matched averaged marginal skewness and kurtosis in a higher order sigma point
filter, and PRD carried this design to scenario generation for portfolio
optimisation, after which \citet{cbh2018} derived the probability allocation
that avoids negative weights in PRD whenever such an allocation exists. GenUT
matches every marginal third and fourth moment and has so far been used in
filtering and uncertainty quantification rather than in scenario generation.
Negative weights recur throughout this literature. The higher order unscented
estimator of \citet{house2021}, as described by \citet{yang2024}, keeps its
last weight nonnegative by raising each normalised kurtosis \(K_j\) to at
least \(n/(1-\delta_H)+S_j^2\) for a control parameter \(\delta_H\), so that its
requirement grows with the dimension, while the square root variant of
\citet{yang2024} accepts a negative weight and absorbs it through a downdate
of the covariance factor, which is admissible in filtering but not in a
stochastic program.
Closed form feasibility conditions for small scenario sets also appear in
operational research, where \citet{geyer2014a} derive bounds under which
scenario sets that match a covariance matrix with the least possible number
of scenarios are free of arbitrage.

The freedom to rotate a sigma point set is well known in nonlinear filtering,
where the matrix square root and its rotation are chosen to improve
integration or filtering accuracy
\citep{julier2002,straka2013,dunik2014,dunik2016}. Random orthogonal matrix
simulation rotates samples with equal probabilities so that they reproduce a
mean and covariance exactly \citep{rom2011}, and the choice of rotation
changes the marginal higher moments of the simulated returns
\citep{ledermann2012}. Within this framework, rotations have been chosen to
target multivariate skewness \citep{hanke2017,amw2022} and to produce
arbitrage free samples by rotating a simplex \citep{geyer2014b}, while
\citet{alexander2024} give necessary and sufficient conditions for matching
the multivariate skewness and kurtosis measures of Kollo and show that a
suitable orthonormal basis simplifies them. \citet{foldnes2016} solve the
same linear systems in entrywise powers for continuous simulation and choose
between the symmetric and Cholesky roots without optimising the factor, and
minimising the sum of fourth powers of a factor is the opposite of the
quartimax criterion of factor analysis \citep{quartimax1954}. To our
knowledge, no previous work chooses the covariance factor so that the
probabilities of a moment matching scenario set are nonnegative, which is the
role the rotation plays in this paper.

Approaches based on probability metrics or on the decision problem itself
address a different setting. Scenario reduction with respect to probability
metrics \citep{pflug2001,dupacova2003} has known fundamental limits and
approximation guarantees \citep{rujeerapaiboon2022}, problem based
generation approximates the expected recourse function directly
\citep{henrion2022}, and optimisation based reduction exploits the structure
of two stage problems \citep{bertsimas2023}. Problem driven methods place
scenarios where tail risk measures are determined \citep{fairbrother2022} or
decompose the output distributions of candidate decisions
\citep{narum2024}, while learning based methods train generative models on
empirical scenarios \citep{pisinger2026} or search for diverse scenario sets
with genetic algorithms \citep{oliveira2022}. These methods require a richer
description of the distribution or access to the decision model, whereas
closed form moment matching serves the common situation in which only a few
moments are trusted and scenario sets must be regenerated frequently, and a
valid moment matched set can also serve as their input.

\section{Moment matching and component feasibility}\label{sec:setting}

Let \(\mu\in\R^n\) be the target mean, let \(\Sigma\) be a positive
definite covariance matrix, and let \(m_3,m_4\in\R^n\) contain the marginal
central moments of orders three and four. We write
\(\sigma_j=\sqrt{\Sigma_{jj}}\), \(S_j=m_{3j}/\sigma_j^3\), and
\(K_j=m_{4j}/\sigma_j^4\), where \(S_j\) and \(K_j\) denote marginal
skewness and raw kurtosis. Every marginal distribution satisfies Pearson's
inequality
\begin{equation}\label{eq:pearson-marg}
  K_j\ge S_j^2+1,
\end{equation}
with equality only for a two point distribution \citep{pearson1916}. This is
a natural benchmark because no scenario generator can reproduce targets that
violate it. For a matrix or vector \(B\), the notation \(B\had{k}\) denotes
its entrywise \(k\)th power.

Every covariance factor has the form \(L=\Sigma^{1/2}Q\), where
\(QQ^\top=I\), so selecting a factor is equivalent to selecting an
orthogonal rotation. The generators considered below place scenarios along
directions \(\ell\) through the mean. A component with probability \(\pi\)
has the form \(X=\mu+T\ell\), where \(T\) is a scalar displacement satisfying
\(\E[T]=0\). It contributes
\(\pi\E[T^2]\ell\ell^\top\) to the covariance and
\(\pi\E[T^k]\ell\had{k}\) to the marginal central moments of order \(k\).
This separates the geometry of the direction from the univariate moments
generated along it.

\subsection{Two particular constructions}

GenUT assigns one component to every column \(L_i\) of a factor satisfying
\(LL^\top=\Sigma\). Each component contributes unit variance along its
column, while \(u_i\) and \(w_i\) denote its third and fourth moment
contributions. Matching all marginal moments requires
\[
  m_3=L\had3u,\qquad m_4=L\had4w.
\]
When \(L\had3\) and \(L\had4\) are nonsingular, the factor uniquely
determines \(u\) and \(w\). We call any mixture of one component along each
column of \(L\), with arbitrary displacement laws, together with an atom at
\(\mu\), a column component distribution. GenUT is the case of two point
displacements, which gives \(2n+1\) scenarios, and although any covariance
factor can be used, the published implementation selects the symmetric
root.

PRD instead match the aggregate marginal moments
\(M_3=\sum_jm_{3j}\) and \(M_4=\sum_jm_{4j}\). They select a direction
\(Z\in\R^n\) and factor the remaining covariance as
\(LL^\top=\Sigma-ZZ^\top\). For every column \(L_i\) and level
\(k=1,\ldots,s\), the method places two symmetric points at
\[
  \mu\pm\frac{L_i}{\sqrt{2sq_k}},
\]
with probability \(q_k>0\) assigned to each point. These pairs have total
probability \(t=2n\sum_kq_k\), which we require to be below one, reproduce
\(LL^\top\), and contribute no
third moment. The remaining probability \(1-t\) is assigned to a three point
component along \(Z\), which reproduces \(ZZ^\top\), carries all of \(M_3\),
and supplies the part of \(M_4\) not used by the symmetric pairs. The
original method chooses \(Z=\rho\sigma\) for a scalar \(\rho>0\), where
\(\sigma\) is the vector of marginal standard deviations. We write \(\zeta_3=\sum_jZ_j^3\),
\(\zeta_4=\sum_jZ_j^4\), and
\(\Lambda=\sum_{i,j}L_{ji}^4\), and assume \(\zeta_3\ne0\).

\subsection{A component feasibility principle}\label{sec:component}

Consider a component of probability \(\pi\in(0,1]\). Let \(T\) be its
displacement and suppose that its contributions to moments of orders two,
three, and four are \(c_k=\pi\E[T^k]\). The conditional variance, skewness,
and kurtosis are then \(c_2/\pi\), \(\sqrt\pi c_3/c_2^{3/2}\), and
\(\pi c_4/c_2^2\), and applying Pearson's inequality to this conditional law
gives the following result.

\begin{lemma}\label{lem:pearson}
Let \(\pi\in(0,1]\), \(c_2>0\), and \(c_3,c_4\in\R\). A mean zero random
variable \(T\) satisfying \(\pi\E[T^k]=c_k\) for \(k=2,3,4\) exists if and
only if
\begin{equation}\label{eq:pearson}
  c_2c_4-c_3^2\ge\frac{c_2^3}{\pi}.
\end{equation}
Equality holds if and only if \(T\) has two support points. Every such
variable satisfies
\begin{equation}\label{eq:offmass}
  \pi\Pr(T\ne0)\ge\frac{c_2^3}{c_2c_4-c_3^2},
\end{equation}
and when \eqref{eq:pearson} holds, a law with two nonzero support points and,
when the inequality is strict, an atom at zero attains the targets with
equality in \eqref{eq:offmass}.
\end{lemma}

The lemma gives a direct mass interpretation because a rare component must
use large displacements to contribute a fixed variance, which raises its
fourth moment contribution and implies that a component required to deliver
relatively little kurtosis must receive more probability. The minimum mass
in \eqref{eq:offmass} consequently depends only on the assigned moment
contributions, and it is attained by the law used in the constructions
below.

The condition is equivalent to positive semidefiniteness of the moment
matrix of \((1,T,T^2)\) and forms a special case of the truncated Hamburger
moment problem \citep{curtofialkow1991}, whose relevant rank condition holds
here because \(c_2>0\). Related moment conditions arise in linear programming
over probability measures \citep{prekopa1990}.

\begin{corollary}\label{cor:unit}
Let \(c_2=1\), \(c_3=u\), and \(c_4=w\). A component exists if and only if
\(w-u^2>0\) and \(\pi\ge1/(w-u^2)\), and its minimum mass away from zero is
\(1/(w-u^2)\).
\end{corollary}

\subsection{Feasibility for a fixed factor}\label{sec:exact}

Once the factor has determined the contributions \(u\) and \(w\) of the
column components, Corollary~\ref{cor:unit} gives the minimum probability
required by every component.

\begin{theorem}\label{thm:blocks}
Let \(L\) be a factor of \(\Sigma\) for which \(L\had3\) and \(L\had4\) are
nonsingular, and define \(u=(L\had3)^{-1}m_3\), \(w=(L\had4)^{-1}m_4\), and
\(d=w-u\had2\).
\begin{enumerate}
\item[(a)] A column component distribution matches
\(\mu,\Sigma,m_3,m_4\) if and only if every component has \(\pi_i>0\),
\(\pi_i\E[T_i^2]=1\), \(\pi_i\E[T_i^3]=u_i\), and
\(\pi_i\E[T_i^4]=w_i\).
\item[(b)] Such a distribution exists if and only if
\begin{equation}\label{eq:feasible}
  d_i>0\quad(i=1,\ldots,n),\qquad
  \sum_{i=1}^n\frac1{d_i}\le1.
\end{equation}
Every matching distribution satisfies \(\Pr(X\ne\mu)\ge\sum_i1/d_i\).
\item[(c)] When \eqref{eq:feasible} holds, let
\(a_i=\frac12(u_i+\sqrt{u_i^2+4d_i})\) and \(b_i=a_i-u_i\). The \(2n+1\)
points \(\mu+a_iL_i\) w.p.\ \(b_i/(d_i(a_i+b_i))\), \(\mu-b_iL_i\) w.p.\
\(a_i/(d_i(a_i+b_i))\), and \(\mu\) w.p.\ \(1-\sum_i1/d_i\) form a matching
distribution that attains the lower bound in part (b).
\end{enumerate}
\end{theorem}

Condition~\eqref{eq:feasible} is a probability budget in which \(d_i\) is the
fourth moment assigned to component \(i\) beyond the amount forced by its
third moment, so the component requires probability \(1/d_i\) and
feasibility requires these minimum probabilities to fit within total mass
one. When \(n=1\), the condition reduces exactly to Pearson's inequality.
Assigning more probability to a component than its minimum only creates
another atom at \(\mu\), so for a fixed factor the construction in part (c)
attains the largest possible centre probability, which only a change of
factor can increase.

We call \(L\) feasible when its third and fourth entrywise powers are
nonsingular and \eqref{eq:feasible} holds.

\begin{remark}
GenUT requires \(m_4\ge L\had4u\had2\) entrywise. Since
\(m_4=L\had4w\), this condition is \(L\had4d\ge0\). It does not imply
\(d>0\) and it does not check the probability budget in
\eqref{eq:feasible}. For example, let \(L\) be the rotation by \(30^\circ\), \(u=0\), and
\(w=d=(-1,40)^\top\), so that \(\Sigma=I\), \(m_3=0\), and
\(m_4=L\had4d=(31/16,359/16)^\top\). Both marginals satisfy Pearson's inequality and the GenUT check passes.
Nevertheless, \(d_1<0\). The first component would require a fourth moment
below its squared third moment, and the point formula contains
\(\sqrt{4w_1-3u_1^2}=\sqrt{-4}\).
\end{remark}

The symmetric pairs in PRD do not encounter the Pearson restriction because
they have zero skewness, whereas the restriction applies to the \(Z\)
component that carries all of \(M_3\). The pairs nevertheless consume part
of \(M_4\), and this consumption increases as their assigned probability
decreases.

Define \(A=M_4/\zeta_4-(M_3/\zeta_3)^2\) and \(c=n\Lambda/\zeta_4\). The
quantity \(A\) is the fourth moment available to the \(Z\) component
beyond its squared third moment before accounting for the pairs, while
\(c\) measures the fourth moment cost of those pairs and contains the entire
dependence on \(L\).

\begin{proposition}\label{prop:prd}
Fix \(Z\) with \(\zeta_3\ne0\) and a factor \(L\) of
\(\Sigma-ZZ^\top\).
\begin{enumerate}
\item[(a)] For fixed \(s\) and \(q\) with \(t<1\), let \(H=\sum_k1/q_k\).
The PRD formulas give real points and nonnegative probabilities if and only
if \(\mathcal D(q)=(1-t)(A-\Lambda H/(2s^2\zeta_4))\ge1\). Whenever this
condition holds, the three points along \(Z\) are the law in
Lemma~\ref{lem:pearson}.
\item[(b)] \(\mathcal D(q)\le(1-t)(A-c/t)\), with equality if and only if
all \(q_k\) are equal.
\item[(c)] A choice of \(q\) producing a valid scenario set exists if and
only if
\begin{equation}\label{eq:prd}
  A>0,\qquad \sqrt A\ge1+\sqrt c.
\end{equation}
Equal values of \(q_k\) are optimal and give total pair probability
\(t^*=\sqrt{c/A}\), and neither the condition nor \(t^*\) depends on
\(s\).
\end{enumerate}
\end{proposition}

The proposition describes the allocation of probability between the two
parts of PRD, since symmetric pairs with total probability \(t\) lie at
distances of order \(1/\sqrt t\) and use \(c/t\) units of the available
fourth moment, while the \(Z\) component retains probability \(1-t\). Its
Pearson margin is therefore \((1-t)(A-c/t)\), which is maximised at \(t^*\)
with value \((\sqrt A-\sqrt c)^2\). The optimal probabilities derived by
\citet{cbh2018} therefore follow directly from the component feasibility
principle, and what the proposition adds is the equivalence in part (a) and
the reduction of the entire dependence on the factor to \(\Lambda\).
The PRD three point law is itself optimal because whenever any law along
\(Z\) can meet the assigned moments, the published formula can do so.

\subsection{How the factor changes feasibility}\label{sec:size}

PRD depends on \(L\) only through \(\Lambda=\sum_{i,j}L_{ji}^4\). Column
components instead depend on the full matrices \(L\had3\) and \(L\had4\),
although their column sums \(\Lambda_i=\sum_jL_{ji}^4\), which add up to
\(\Lambda\),
provide useful feasibility bounds and allow the next proposition to
describe the range of \(\Lambda\) over all covariance factors.

\begin{proposition}\label{prop:floor}
Let \(R\) be an \(n\) by \(n\) positive definite matrix and let
\(LL^\top=R\). Then
\[
  \frac1n\sum_jR_{jj}^2
  \le\sum_{i,j}L_{ji}^4
  \le\sum_jR_{jj}^2.
\]
The lower bound holds with equality if and only if
\(|L_{ji}|=\sqrt{R_{jj}/n}\) for every \(j,i\). The upper bound holds with equality if and only if \(R\)
is diagonal and \(L\) is a signed and scaled permutation matrix.
\end{proposition}

\begin{proof}
Every row \(j\) satisfies \(\sum_iL_{ji}^2=R_{jj}\), so the Cauchy Schwarz
inequality gives \(R_{jj}^2=(\sum_iL_{ji}^2)^2\le n\sum_iL_{ji}^4\), with
equality precisely when the squared entries in the row are equal, and
summing over rows proves the lower bound. The identity
\(R_{jj}^2=\sum_iL_{ji}^4+2\sum_{i<k}L_{ji}^2L_{jk}^2\) proves the upper
bound. Equality requires one nonzero entry in every row.
Nonsingularity then places these entries in distinct columns, which gives
the stated form of \(L\) and makes \(R\) diagonal. The converse statements
follow directly.
\end{proof}

The two bounds have a clear geometric meaning because each asset is
concentrated on one column at the upper bound and spread evenly across all
columns at the lower bound, with the two values differing by a factor
\(n\). Symmetric and Cholesky roots of moderately correlated covariance
matrices usually concentrate row mass on a few columns, so that their ratio
to the lower bound grows nearly in proportion to \(n\).

Another common choice is the eigenvector factor \(E\,\Omega^{1/2}\) from
the spectral decomposition \(E\Omega E^\top\) of a covariance matrix, which for two
assets gives the principal directions studied in Section~\ref{sec:two}. On
the synthetic covariance matrices used below, its
median value of \(\Lambda\) falls from \(0.48n\) times the lower bound at
\(n=5\) to \(0.13n\) at \(n=160\). The corresponding ratios for the
symmetric root are \(0.90n\) and \(0.59n\). The same fourth powers also
bound the probability required by column components.

\begin{proposition}\label{prop:massbound}
Let \(L\) be a factor of \(\Sigma\) for which \(L\had3\) and \(L\had4\) are
nonsingular and \(d>0\). Let \(M_4=\sum_jm_{4j}\). Then
\begin{equation}\label{eq:massbound}
  \sum_i\frac1{d_i}
  \ge\frac{\left(\sum_i\sqrt{\Lambda_i}\right)^2}{M_4}
  \ge\frac{\left(\sum_iL_{ii}^2\right)^2}{M_4}.
\end{equation}
A feasible factor must therefore satisfy
\((\sum_i\sqrt{\Lambda_i})^2\le M_4\).
\end{proposition}

\begin{proof}
Summing \(m_4=L\had4w\) over assets gives \(M_4=\sum_iw_i\Lambda_i\).
Since \(0<d_i=w_i-u_i^2\le w_i\), the Cauchy Schwarz inequality gives
\[
  \left(\sum_i\sqrt{\Lambda_i}\right)^2
  =\left(\sum_i\sqrt{w_i\Lambda_i}\frac1{\sqrt{w_i}}\right)^2
  \le M_4\sum_i\frac1{w_i}
  \le M_4\sum_i\frac1{d_i}.
\]
The second bound follows from \(\Lambda_i\ge L_{ii}^4\).
\end{proof}

This result explains the dimension effect for column components by
considering uncorrelated assets with equal variances under the diagonal
factor, for which \(\Lambda_i=\sigma_i^4\) and \eqref{eq:massbound} requires
the pooled kurtosis to satisfy \(\bar K=M_4/\sum_j\sigma_j^4\ge n\), whereas a factor attaining the lower bound of Proposition~\ref{prop:floor}
reduces the same requirement to \(\bar K\ge1\). When some \(d_i\le0\),
infeasibility already follows from Theorem~\ref{thm:blocks}, while otherwise
the diagonal bound in \eqref{eq:massbound} can provide the certificate. For
PRD, the effect can be stated exactly when \(Z=\rho\sigma\).

\begin{corollary}\label{cor:diagZ}
Let \(Z=\rho\sigma\) with \(\rho>0\), and define \(R=\Sigma-ZZ^\top\), the
pooled kurtosis \(\bar K=M_4/\sum_j\sigma_j^4\), and the pooled skewness
\(\bar S=M_3/\sum_j\sigma_j^3\).
\begin{enumerate}
\item[(i)] The matrix \(R\) is positive definite if and only if \(\rho\) is
below \(\rho_{\max}=(\sigma^\top\Sigma^{-1}\sigma)^{-1/2}\).
\item[(ii)] For every factor \(L\) of \(R\), feasibility for some \(q\)
requires
\begin{equation}\label{eq:rhomin}
  \rho^2(\bar K-1)\ge\bar S^2.
\end{equation}
\item[(iii)] Write \(\Lambda=(1+\eta)\sum_jR_{jj}^2/n\) with
\(\eta\ge0\). Feasibility for some \(q\) holds if and only if
\begin{equation}\label{eq:attained}
  A>0,\qquad
  \sqrt A\ge
  1+\sqrt{1+\eta}\frac{1-\rho^2}{\rho^2}.
\end{equation}
When \(\eta=0\), condition \eqref{eq:attained} is equivalent to
\eqref{eq:rhomin}.
\end{enumerate}
\end{corollary}

For \(\bar K>1\), parts (i) and (ii) define the factor independent interval
\(\bar S^2/(\bar K-1)\le\rho^2<\rho_{\max}^2\). The lower endpoint prevents a short direction \(Z\) from carrying all the
pooled skewness with insufficient kurtosis, while the upper endpoint ensures
that the remaining covariance is positive definite. Part (iii) measures how
much of this interval a particular factor retains, with smaller values of
\(\eta\) preserving a larger portion.

\begin{remark}\label{rem:floor}
Equality in the lower bound requires signs \(s_{ji}\in\{-1,1\}\) such that
\(R_{jk}/\sqrt{R_{jj}R_{kk}}=n^{-1}\sum_is_{ji}s_{ki}\) for all \(j,k\). A
generic covariance matrix does not have this structure, and for \(R=I\)
the condition requires a scaled Hadamard matrix, which exists only in
selected dimensions. Flat orthogonal matrices and approximate constructions
are studied by \citet{jaming2015}, and we call a rotation that minimises
\(\Lambda\) a minimum quartic factor.
\end{remark}

\subsection{Rotated moment matching}\label{sec:method}

The preceding results add one numerical step to closed form moment matching.
Starting from a fixed factor \(L_0\), we write \(L=L_0Q\) and search over
orthogonal \(Q\) with the gradient projection method of
\citet{jennrich2001}. Once a feasible factor is found, the scenarios follow
directly from Theorem~\ref{thm:blocks} or Proposition~\ref{prop:prd}.

For column components, the objective is the total component probability
\(F(Q)=\sum_i1/d_i(L_0Q)\), which equals one minus the centre weight. Each
term is continued linearly once \(d_i\) falls below \(10^{-3}\), so that the
search can move through infeasible factors, and each search stops as soon as
\(F(Q)\le0.9\) or the iteration stalls. We use the symmetric root and \(J\)
random rotations \(\exp(h(G_j-G_j^\top)/\sqrt{2n})\) as starting points,
where the entries of \(G_j\) are independent standard normal variables and
\(h\) controls the spread around the identity, and we return the factor with
the smallest \(F\), which by Theorem~\ref{thm:blocks} leaves the largest
centre probability among the factors found.

For PRD, Corollary~\ref{cor:diagZ} first determines the factor independent
interval for \(\rho\), whose lower endpoint is well defined because
\(\bar K\ge1\) by Pearson's inequality. We test three interior values and
compute a minimum quartic factor for each residual covariance. Feasibility
is then checked using Proposition~\ref{prop:prd}, with equal pair
probabilities because they are optimal.

\begin{algorithm}[!htbp]
\caption{Rotated moment matching}\label{alg:rmm}
\begin{algorithmic}[1]
\Require targets \(\mu,\Sigma,m_3,m_4\), random starts \(J\), spread
\(h\), PRD levels \(s\)

\Statex \textit{Column component variant}
\State \(L_0\gets\Sigma^{1/2}\) and
\(\mathcal Q_0\gets\{I,Q_1,\ldots,Q_J\}\) with
\(Q_j=\exp(h(G_j-G_j^\top)/\sqrt{2n})\)
\For{\(Q_0\in\mathcal Q_0\)}
  \State minimise the linearly continued \(F(Q)=\sum_i1/d_i(L_0Q)\) from
  \(Q_0\) over orthogonal \(Q\) until \(F(Q)\le0.9\) or the iteration stalls
  \State retain \(Q\) if \(d_i(L_0Q)>0\) for all \(i\) and \(F(Q)\le1\)
\EndFor
\State \textbf{if} no factor is retained \textbf{then return} no feasible
factor found
\State \Return the scenarios of Theorem~\ref{thm:blocks}(c) for \(L=L_0Q\),
with \(Q\) the retained rotation of smallest \(F\)

\Statex \textit{PRD variant}
\State compute \(\rho_{\max}\) and \(\rho_{\min}=|\bar S|/\sqrt{\bar K-1}\),
with \(\rho_{\min}=0\) when \(\bar S=0\), from Corollary~\ref{cor:diagZ}
\State \textbf{if} \(\rho_{\min}\ge\rho_{\max}\) \textbf{then return} no
feasible factor exists for \(Z=\rho\sigma\)
\For{\(r=1,2,3\)}
  \State \(\rho\gets\rho_{\min}+r(\rho_{\max}-\rho_{\min})/4\)
  \State compute a minimum quartic factor \(L\) of
  \(\Sigma-\rho^2\sigma\sigma^\top\)
  \State \textbf{if} \(\sqrt A\ge1+\sqrt c\) \textbf{then return} the PRD
  scenarios with \(q_k=t^*/(2ns)\)
\EndFor
\State \Return no feasible factor found
\end{algorithmic}
\end{algorithm}

The method does not modify the targets, so every returned scenario set
matches the prescribed moments exactly. An empty PRD interval certifies that
no factor and no value of \(\rho\) with \(Z=\rho\sigma\) is feasible. Other
unsuccessful outcomes remain inconclusive because the rotation searches are
local, although Section~\ref{sec:two} shows that feasible factors need not
always exist.

\section{Computational experiments}\label{sec:numerics}

We use the moments of twenty FTSE stocks reported in Tables A.1 to A.3 of
PRD, the moments of the three year sample of 35 Chilean IPSA stocks used by
\citet{cbh2018}, daily prices of IPSA stocks from 2005 to 2015, and 144
synthetic datasets with dimensions from 5 to 160. The IPSA prices give 33
rolling windows of two years of weekly returns on 33 stocks and 28 rolling
windows of three years of daily returns on 31 stocks. The synthetic returns follow a one factor
model with standardised Johnson SU shocks under three combinations of
skewness and tail thickness, with ten datasets for each combination and
dimension up to 40 and four for the dimensions 80 and 160, and their
targets are the moments of 2000 simulated observations.

All rotated factors are produced by Algorithm~\ref{alg:rmm} with \(J=4\)
random starts of spread \(h=0.3\), and the code that reproduces every result
will be made publicly available in the repository
\texttt{rotated-moment-matching}.

Figure~\ref{fig:genut-real} reports the centre weight of GenUT with the
symmetric root on every rolling IPSA window. The weight is negative
throughout, ranging from \(-4.42\) to \(-1.19\) on the weekly windows and
from \(-3.24\) to \(-1.00\) on the daily windows, and it equals \(-0.99\) on
the FTSE moments and \(-1.39\) on the 2010 IPSA sample, although on the FTSE
moments the componentwise check of GenUT is satisfied. A rotated factor
gives valid probabilities on every one of these samples
(Table~\ref{tab:rotation-gain}), so on these data the failure of the
published construction is caused by its choice of factor alone.

\begin{figure}[!htbp]
\centering
\includegraphics[scale=0.8]{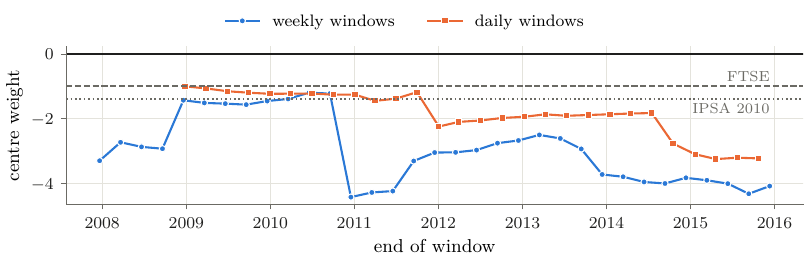}
\caption{Centre weight of GenUT with the symmetric root on rolling IPSA
windows. Horizontal lines mark the FTSE and 2010 IPSA samples.}
\label{fig:genut-real}
\end{figure}

Figure~\ref{fig:prd-real} plots the margin \(\sqrt A-\sqrt c\) of
Proposition~\ref{prop:prd} against \(\rho\) for the three factors, with the
interval of Corollary~\ref{cor:diagZ} shaded and feasibility above the
dashed line at one. The minimum quartic factor is feasible on almost the
whole interval, from \(\rho=0.13\) to \(\rho_{\max}=0.514\) on the FTSE
moments and from \(\rho=0.15\) to \(\rho_{\max}=0.476\) on the IPSA sample,
whereas the Cholesky factor is feasible on the FTSE moments only for
\(\rho\ge0.43\) and the symmetric root only in a narrow band just below
\(\rho_{\max}\), and neither is feasible on the IPSA sample. The values used
in the examples of PRD, marked by dotted lines, are not feasible on their
own FTSE moments, since \(\Sigma-ZZ^\top\) is indefinite at \(\rho=0.7\) and
the symmetric root gives \(\sqrt A-\sqrt c=-0.23<1\) at \(\rho=0.45\).

\begin{figure}[!htbp]
\centering
\includegraphics[scale=0.8]{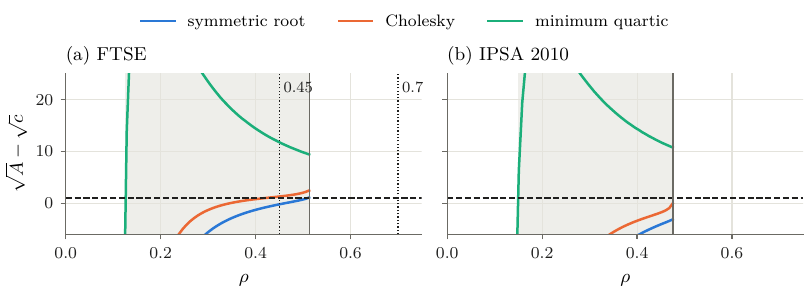}
\caption{PRD margin against \(\rho\) for three factors. The shaded interval
is the window of Corollary~\ref{cor:diagZ}, and feasibility requires values
above one.}
\label{fig:prd-real}
\end{figure}

Figure~\ref{fig:core-scaling} summarises the synthetic study. Panel (a)
shows that \(\Lambda\) for the symmetric and Cholesky roots of the PRD
remainder grows almost in proportion to the upper bound \(n\) of
Proposition~\ref{prop:floor}, while the minimum quartic factor stays at a
median within \(1.07\) times and a maximum within \(1.10\) times the lower
bound at every dimension. Panel (b) shows the
consequence for GenUT, whose symmetric root requires a total component
probability that grows from a median of \(1.1\) with five assets to \(23.4\)
with 160 assets and exceeds one on every dataset from ten assets on. Panel
(c) reports the share of datasets with a valid scenario set when \(\rho\)
is chosen on a fine grid. PRD with the symmetric root is feasible on every
dataset with five assets but on only two of the thirty with ten, the
Cholesky factor is feasible on nine of those thirty, including both of
them and eight heavy tailed datasets, and neither is feasible on any larger
dataset, whereas the minimum quartic factor
makes PRD feasible on all 144 datasets. The rotation
search finds a feasible factor for column components with the four random
starts on 142 datasets, and the remaining two light tailed instances with 160
assets are solved when it is rerun with twelve random starts, one of them
with the smaller spread \(h=0.1\).

\begin{figure}[!htbp]
\centering
\includegraphics[scale=0.8]{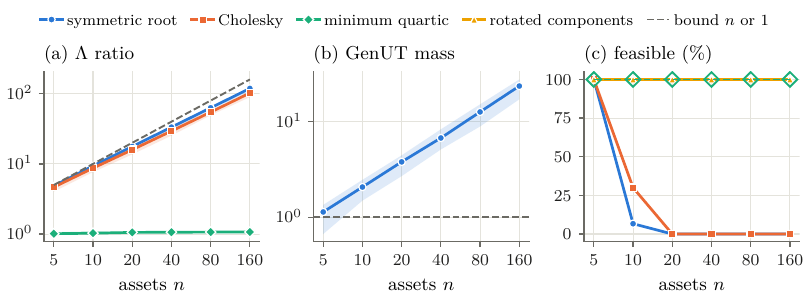}
\caption{Synthetic study. (a) \(\Lambda\) over its lower bound for the PRD
remainder at the middle of the interval of Corollary~\ref{cor:diagZ}, and
(b) total component probability of the GenUT symmetric root, both as medians
with ranges. (c) Share of datasets with a valid scenario set, for PRD under
each factor and for rotated column components.}
\label{fig:core-scaling}
\end{figure}

Table~\ref{tab:rotation-gain} counts, on all 63 real and 144 synthetic
datasets, the scenario sets with valid probabilities under the standard
factor of each generator and under the factor selected by
Algorithm~\ref{alg:rmm}. GenUT with the symmetric root is valid on none of
the real datasets and on only ten synthetic datasets with five assets,
whereas the rotated factor makes it valid on all 207 datasets while keeping
every target moment. PRD with the better of the symmetric and Cholesky
factors, and \(\rho\) chosen on a fine grid, is valid on 65 datasets, namely
the FTSE moments, 25 of the 61 rolling IPSA windows, and 39 synthetic
datasets with at most ten assets, whereas the minimum quartic factor makes
it valid on 206. The single exception is the weekly IPSA window whose
interval in Corollary~\ref{cor:diagZ} is empty, so that no factor can make
PRD feasible there with \(Z=\rho\sigma\). The improvement is largest in high dimension, where the standard
factors fail on every dataset and the rotated factors succeed on all of
them.

\begin{table}[!htbp]
\centering
\small
\input{tab_rotation_gain}
\caption{Datasets with valid probabilities under the standard factor and under the rotated factor of Algorithm~\ref{alg:rmm}.}
\label{tab:rotation-gain}
\end{table}

\FloatBarrier
\section{The case of two assets}\label{sec:two}

The experiments show that a feasible rotation exists on almost every dataset
tested, although the fixed factor conditions do not by themselves establish
when such a factor must exist. With two assets the complete factor class is
indexed by one angle, which makes the existence question tractable and
reveals the distinct roles of skewness and kurtosis.

\subsection{Angular representation}

Positive diagonal scaling leaves \(u\), \(w\), and \(d\) unchanged
(Lemma~\ref{lem:two-scale}). We may
therefore take unit marginal variances, so that \(\Sigma\) has unit
diagonal and off diagonal entry \(r\) with \(|r|<1\), and write
\(\delta=\arccos r\). For \(\tau\in\R\), define
\begin{equation}\label{eq:two-Lt}
  L(\tau)=
  \begin{pmatrix}
  \cos(\tau+\delta)&\sin(\tau+\delta)\\
  \cos\tau&\sin\tau
  \end{pmatrix}.
\end{equation}
Since its rows are unit vectors with inner product \(r\), this matrix is a
factor of \(\Sigma\).

\begin{lemma}\label{lem:two-angle}
Every factor of \(\Sigma\) equals \(L(\tau)\) for some
\(\tau\in[0,\pi)\), up to column order and column signs. These operations do
not affect feasibility. The matrix \(L(\tau)\had3\) is always nonsingular.
The matrix \(L(\tau)\had4\) is singular only when its columns are the
principal directions of \(\Sigma\), at \(\tau_0=(\pi-\delta)/2\) or
\(\tau_0+\pi/2\).
\end{lemma}

For two components, condition \eqref{eq:feasible} is equivalent to
\begin{equation}\label{eq:pair-form}
  d_1>1,\qquad d_2>1,\qquad
  (d_1-1)(d_2-1)\ge1.
\end{equation}
It follows in particular that \(d_1+d_2\ge4\).

\subsection{Symmetric marginals}

When the marginal skewness is zero, a rotation reaches the natural
univariate bound.

\begin{theorem}\label{thm:two-sym}
Let \(n=2\), \(|r|<1\), and \(S=0\). A feasible factor exists if and only if
\begin{equation}\label{eq:two-symcond}
  K_1\ge1,\qquad K_2\ge1,\qquad
  (K_1,K_2)\ne(1,1).
\end{equation}
More precisely, the following statements hold whenever
\eqref{eq:two-symcond} is satisfied.
\begin{enumerate}
\item[(a)] If \(K_1<K_2\), then some \(L(\tau)\) is feasible with
\(d_1=K_1/\cos^2(\tau+\delta)\), \(d_2=K_1/\sin^2(\tau+\delta)\), and
\(\sum_i1/d_i=1/K_1\).
\item[(b)] If \(K_1=K_2=K>1\), then for every
\(\kappa\in(0,\min\{1-r^2,K-1\}]\) some \(L(\tau)\) is feasible with
\(d_i\ge K/(1+\kappa)\) and \(\sum_i1/d_i=(1+\kappa)/K\).
\end{enumerate}
\end{theorem}

Pearson's inequality reduces to \(K_j\ge1\) in this setting, and Theorem
\ref{thm:two-sym} shows that rotation reaches this bound. The exceptional
case \(K=(1,1)\) remains attainable but requires a factor for which
\(L\had4\) is singular (Remark~\ref{rem:two-11}), while in part (a) the two
components together need the same probability as a single marginal with the
smaller kurtosis and in part (b) their total probability can be made
arbitrarily close to \(1/K\). The symmetric root, by contrast, can require
substantially more kurtosis.

\begin{proposition}\label{prop:two-symroot}
Let \(n=2\), \(|r|<1\), and \(K_1=K_2=K\). If \(S=0\), the symmetric root
is feasible if and only if \(K\ge2-r^2\). For arbitrary \(S\), it is
feasible whenever
\(2K/(2-r^2)-\|S\|_\infty^2/((1-|r|)(1+|r|/2)^2)\ge2\).
\end{proposition}

At \(r=0\), the symmetric root is the identity, so each column carries one
asset and the minimum component probabilities add. For \(n\) uncorrelated
assets, this factor therefore requires \(\sum_j1/(K_j-S_j^2)\le1\). With symmetric marginals, the harmonic mean of the kurtoses must be at least
\(n\). With equal marginals the requirement reads \(K\ge n+S^2\), which is the
threshold that the higher order unscented estimator imposes on normalised
kurtosis when \(\delta_H=0\) \citep{house2021,yang2024}.
Proposition~\ref{prop:massbound} traces this requirement to the
concentration of the identity factor, since a factor at the lower bound of
Proposition~\ref{prop:floor} relaxes the corresponding necessary condition
to \(\bar K\ge1\).

\subsection{The skewness tradeoff}

With skewed marginals, rotation need not attain Pearson's bound, even when
the target itself is attainable.

\begin{proposition}\label{prop:two-fail}
Let \(n=2\), \(r=0\), \(S=(1,1)\), and
\(K=(5/2,5/2)\). No factor of \(I\) is feasible. The target is nevertheless
attainable, since two independent copies of the distribution with
\(Y=(1+\sqrt7)/2\) w.p.\ \((\sqrt7-1)/(3\sqrt7)\), \(Y=-(\sqrt7-1)/2\) w.p.\
\((\sqrt7+1)/(3\sqrt7)\), and \(Y=0\) w.p.\ \(1/3\) have mean zero, covariance \(I\), skewness \((1,1)\), and kurtosis
\((5/2,5/2)\).
\end{proposition}

Because each marginal has positive slack above Pearson's bound, the
obstruction is created by the column structure rather than by unattainable
marginal moments.

For \(r=0\), \(L(\tau)\) is an ordinary rotation, and we define
\(y=\cos^2\tau\sin^2\tau\in[0,1/4]\). The value \(y=0\) corresponds to the coordinate directions, whereas
\(y=1/4\) corresponds to a rotation of \(45^\circ\), and when both marginal
kurtoses equal \(K\), Proposition~\ref{prop:two-skew-rzero} in the appendix
gives
\begin{equation}\label{eq:two-tradeoff}
  d_1+d_2
  =\frac{2K}{1-2y}
  -\frac{\|S\|_2^2}{1-3y}.
\end{equation}
Moving toward \(45^\circ\) doubles the fourth moment allocated to each
component while multiplying the squared skewness allocation by four, so the
first effect improves feasibility and the second reduces it.

For the target in Proposition~\ref{prop:two-fail}, where \(\|S\|_2^2=2\)
and \(K=5/2\), the second effect dominates at every angle, since
\eqref{eq:two-tradeoff} remains below four for all values of \(y\) and the
pair condition \eqref{eq:pair-form} therefore fails.

\subsection{A sufficient condition for arbitrary correlation}

For \(r\ne0\), the principal directions replace the \(45^\circ\) rotation.
At \(\tau_0\), the columns are \(L_1=\sqrt{(1-r)/2}\,(-1,1)^\top\) and
\(L_2=\sqrt{(1+r)/2}\,(1,1)^\top\). Although \(L(\tau_0)\had4\) is singular, the vector \(d(\tau)\) has a finite
limit whenever both marginal kurtoses are equal.

\begin{theorem}\label{thm:two-skew-pca}
Let \(n=2\), \(|r|<1\), \(S\in\R^2\), and \(K_1=K_2=K\). Define \(\gamma_-=(S_1-S_2)/(1-r)\) and \(\gamma_+=(S_1+S_2)/(1+r)\). If
\(K>\max\{\gamma_-^2,\gamma_+^2\}\) and
\begin{equation}\label{eq:two-skew-pca}
  \frac{1-r}{2(K-\gamma_-^2)}
  +\frac{1+r}{2(K-\gamma_+^2)}
  <1,
\end{equation}
then \(L(\tau)\) is feasible for every \(\tau\ne\tau_0\) sufficiently close
to \(\tau_0\).
\end{theorem}

The two terms in \eqref{eq:two-skew-pca} are the limiting component
probabilities along the principal directions. With \(S=0\), their sum is
\(1/K\), which agrees with Theorem~\ref{thm:two-sym}.

\subsection{Numerical illustration}

Figure~\ref{fig:two-asset} plots the centre weight \(1-\sum_i1/d_i\) of the
column construction against the angle \(\tau\) for three two asset targets.
With symmetric marginals in panel (a), the weight is positive near the
principal directions and approaches \(1-1/K=1/3\) there, as
Theorem~\ref{thm:two-sym} predicts, while the symmetric root gives \(-0.17\)
in agreement with the threshold \(K\ge2-r^2=1.75\) of
Proposition~\ref{prop:two-symroot}. For the target of
Proposition~\ref{prop:two-fail} in panel (b), the weight remains negative at
every angle, with a best value of \(-0.17\). Panel (c) uses the target of
Figure~\ref{fig:geometry}, which satisfies condition~\eqref{eq:two-skew-pca},
and the weight is positive in a neighbourhood of the principal directions,
as Theorem~\ref{thm:two-skew-pca} asserts. In panels (a) and (c), whose two
marginals are identical, the symmetric root coincides with the least
favourable rotation.

\begin{figure}[!htbp]
\centering
\includegraphics[scale=0.8]{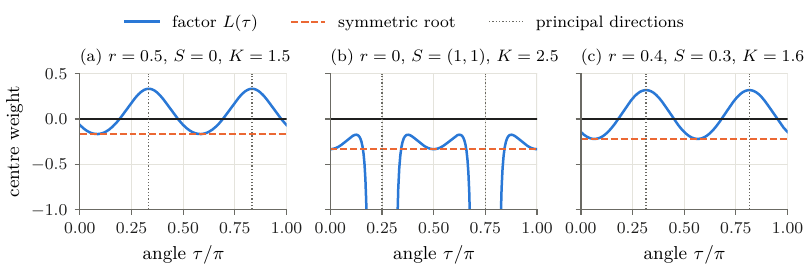}
\caption{Centre weight of the column construction against the rotation
angle for three targets. Dotted lines mark the principal directions.}
\label{fig:two-asset}
\end{figure}

\FloatBarrier
\section{Conclusion}\label{sec:conclusion}

This paper shows that the covariance factor is a design choice rather than a
fixed input in closed form moment matching. By rotating this factor, a
generator can preserve the target moments and the number of scenarios while
turning otherwise invalid weights into a probability distribution.

Pearson's inequality with component mass provides exact feasibility
conditions for GenUT and PRD, while the entrywise fourth powers of the factor
explain the effect of dimension and the advantage of diffuse rotations. The
two asset analysis also clarifies how rotation balances skewness and
kurtosis. The numerical illustrations confirm these mechanisms and produce a
feasible rotation on every dataset considered, including problems with up to
160 assets, except one on which no factor can make PRD feasible.

These results support a simple practical approach in which feasibility is
checked for the selected factor and rotation is considered before the target
moments are modified.

\section*{Funding}

This work was supported by the Agencia Nacional de Investigaci\'on y
Desarrollo (ANID), Chile, through FONDECYT Postdoctoral grant 3240505.







\appendix

\section{Proofs}\label{app:proofs}

\subsection{The component feasibility principle}\label{app:pearson}

\begin{proof}[Proof of Lemma~\ref{lem:pearson}]
Write \(v=\E[T^2]=c_2/\pi\), \(\mu_3=\E[T^3]=c_3/\pi\), and
\(\mu_4=\E[T^4]=c_4/\pi\), so that condition~\eqref{eq:pearson} reads
\(v\mu_4-\mu_3^2-v^3\ge0\). The moment matrix of \((1,T,T^2)\) has rows \((1,0,v)\), \((0,v,\mu_3)\),
and \((v,\mu_3,\mu_4)\). It is positive semidefinite and its determinant is
\(v\mu_4-\mu_3^2-v^3\), which proves necessity. Equality holds precisely
when \(1,T,T^2\) are linearly dependent in \(L^2\), which means that \(T\) is
supported on the roots of a nonzero polynomial of degree at most two. Since
\(v>0\) excludes a constant \(T\), the polynomial has degree two and \(T\)
has exactly two support points, while conversely a two point law satisfies
\((T-a)(T+b)=0\) and therefore gives a singular moment matrix.

For the construction, let \(\vartheta=\mu_3/v\),
\(a=\frac12(\vartheta+\sqrt{\vartheta^2+4v})\), and \(b=a-\vartheta\). Both
\(a\) and \(b\) are positive, and the law w.p.\ \(b/(a+b)\) at
\(a\) and \(a/(a+b)\) at \(-b\) has mean zero, \(\E[T^2]=ab=v\),
\(\E[T^3]=ab(a-b)=\mu_3\), and \(\E[T^4]=ab((a-b)^2+ab)=\mu_3^2/v+v^2\), so it
attains equality in the condition. In the general case,
define \(\lambda=v^3/(v\mu_4-\mu_3^2)\in(0,1]\) and use this two point law
with second and third moments \(v/\lambda\) and \(\mu_3/\lambda\) on an
event of probability \(\lambda\), placing the remaining probability at zero.
The second and third moments remain \(v\) and \(\mu_3\), while the fourth
moment becomes
\(\lambda(\mu_3^2/(\lambda v)+v^2/\lambda^2)=(\mu_3^2+v^3/\lambda)/v=\mu_4\).
Its mass away from zero is
\(\pi\lambda=\pi v^3/(v\mu_4-\mu_3^2)=c_2^3/(c_2c_4-c_3^2)\), which is
equality in \eqref{eq:offmass}. For the bound itself, note that
\(\Pr(T\ne0)>0\) because \(c_2>0\), and that conditioning on \(T\ne0\)
preserves the contributions \(c_k\) while replacing \(\pi\) by
\(\pi'=\pi\Pr(T\ne0)\). The necessity part applied to the conditional law
gives \(c_2c_4-c_3^2\ge c_2^3/\pi'>0\), which is \eqref{eq:offmass}.
\end{proof}

\begin{proof}[Proof of Corollary~\ref{cor:unit}]
With \(c_2=1\), condition~\eqref{eq:pearson} reads \(w-u^2\ge1/\pi\), which
forces \(w-u^2>0\), and \eqref{eq:offmass} gives the minimum mass together
with the law that attains it.
\end{proof}

\subsection{Column components}\label{app:blocks}

\begin{proof}[Proof of Theorem~\ref{thm:blocks}]
Let \(c_{k,i}=\pi_i\E[T_i^k]\) for \(k=2,3,4\), with \(c_{k,i}=0\)
when \(\pi_i=0\), and let \(c_k\) be the vector with entries \(c_{k,i}\).
Every displacement has mean zero, so \(\E[X]=\mu\), and the central moments
are
\begin{gather*}
  \E[(X-\mu)(X-\mu)^\top]=L\operatorname{diag}(c_2)L^\top,\\
  \E[(X_j-\mu_j)^3]=(L\had3c_3)_j,\qquad
  \E[(X_j-\mu_j)^4]=(L\had4c_4)_j.
\end{gather*}
The covariance equals \(LL^\top\) if and only if \(c_2=\mathbf1\), which
also gives \(\pi_i>0\) for every \(i\), while nonsingularity of \(L\had3\)
and \(L\had4\) gives \(c_3=u\) and \(c_4=w\). This proves part (a).

For part (b), let \(C\) denote the index of the component from which \(X\)
is drawn. Component \(i\) has contributions \((1,u_i,w_i)\), so
Corollary~\ref{cor:unit} gives \(d_i>0\) and \(\Pr(C=i,T_i\ne0)\ge1/d_i\).
These events are disjoint and their union is \(\{X\ne\mu\}\), hence
\(\sum_i1/d_i\le\Pr(X\ne\mu)\le1\), which proves necessity together with
the probability bound.

For sufficiency and part (c), set \(\pi_i=1/d_i\) and
\(\pi_0=1-\sum_i1/d_i\), and let \(T_i\) follow the two point law of
Lemma~\ref{lem:pearson} with conditional variance \(d_i\) and conditional
third moment \(u_id_i\), so that \(\vartheta=u_i\), \(a_i-b_i=u_i\), and
\(a_ib_i=d_i\), which are the formulas of part (c). The contributions of
component \(i\) are \(\pi_i\E[T_i^2]=1\), \(\pi_i\E[T_i^3]=u_i\), and
\(\pi_i\E[T_i^4]=u_i^2+d_i=w_i\), so part (a) yields every target moment
while assigning probability \(\sum_i1/d_i\) away from \(\mu\).
\end{proof}

\subsection{PRD}\label{app:prd}

The PRD component along \(Z\) places \(T\) on
\(\{0,\alpha/\sqrt\pi,-\beta/\sqrt\pi\}\) with conditional probabilities
\(\omega_0=1-1/(\alpha\beta)\), \(\omega_1=1/(\alpha(\alpha+\beta))\), and
\(\omega_2=1/(\beta(\alpha+\beta))\), where
\begin{gather*}
  \phi_1=\frac{M_3\sqrt\pi}{\zeta_3},\qquad
  \phi_2=\frac{\pi}{\zeta_4}\left(M_4-\frac{\Lambda H}{2s^2}\right),\\
  \alpha,\beta=\frac12\left(\pm\phi_1+\sqrt{4\phi_2-3\phi_1^2}\right).
\end{gather*}
A direct computation gives
\(\phi_2-\phi_1^2=\pi(A-\Lambda H/(2s^2\zeta_4))=\mathcal D(q)\) with
\(\pi=1-t\).

\begin{proof}[Proof of Proposition~\ref{prop:prd}]
The pairs are symmetric about \(\mu\), so their covariance is
\(\sum_i\sum_k2q_kL_iL_i^\top/(2sq_k)=LL^\top\), their third moments vanish,
and their fourth moments summed over coordinates equal
\(\sum_{j,i,k}2q_kL_{ji}^4/(4s^2q_k^2)=\Lambda H/(2s^2)\). The \(Z\)
component must therefore satisfy \(\pi\E[T]=0\), \(\pi\E[T^2]=1\),
\(\pi\E[T^3]=M_3/\zeta_3\), and
\(\pi\E[T^4]=(M_4-\Lambda H/(2s^2))/\zeta_4\), and Lemma~\ref{lem:pearson}
shows that such a component exists if and only if
\(\pi(A-\Lambda H/(2s^2\zeta_4))=\mathcal D(q)\ge1\). When
\(\mathcal D(q)\ge1\), the quantities \(\alpha\) and \(\beta\) are real and
positive with \(\alpha\beta=\mathcal D(q)\) and \(\alpha-\beta=\phi_1\), and
the three point law above is the construction of Lemma~\ref{lem:pearson} for
these moments. Conversely, \(\omega_1,\omega_2\ge0\) forces
\(\alpha,\beta>0\), and \(\omega_0\ge0\) then forces
\(\mathcal D(q)=\alpha\beta\ge1\). This proves part (a).

For part (b), \(H\ge s^2/\sum_kq_k=2ns^2/t\) by the arithmetic harmonic
mean inequality, with equality exactly when all \(q_k\) are equal, and
substitution in \(\mathcal D(q)\), whose factor \(1-t\) is positive, yields
the bound.

For part (c), let \(G(t)=(1-t)(A-c/t)\) on \(0<t<1\). When \(A\le c\), the
function \(G\) is negative, whereas for \(A>c\) it is strictly concave with
\(G''(t)=-2c/t^3\), has the unique maximiser \(t^*=\sqrt{c/A}\in(0,1)\), and
attains \(G(t^*)=(\sqrt A-\sqrt c)^2\). By part (b), some \(q\) gives
\(\mathcal D(q)\ge1\) if and only if \(A>0\) and \(\sqrt A\ge1+\sqrt c\),
with equal \(q_k\) attaining the maximum. Since \(c=n\Lambda/\zeta_4\) does
not involve \(s\), neither the condition nor \(t^*\) depends on \(s\).
\end{proof}

\begin{proof}[Proof of Corollary~\ref{cor:diagZ}]
For part (i), write
\(\Sigma-\rho^2\sigma\sigma^\top=\Sigma^{1/2}(I-\rho^2\xi\xi^\top)\Sigma^{1/2}\)
with \(\xi=\Sigma^{-1/2}\sigma\), which is positive definite if and only if
\(\rho^2\xi^\top\xi<1\).

For part (ii), \(R_{jj}=(1-\rho^2)\sigma_j^2\),
\(\zeta_3=\rho^3\sum_j\sigma_j^3\), and \(\zeta_4=\rho^4\sum_j\sigma_j^4\),
so \(A=\bar K/\rho^4-\bar S^2/\rho^6\), while Proposition~\ref{prop:floor}
gives \(c\ge\sum_jR_{jj}^2/\zeta_4=((1-\rho^2)/\rho^2)^2\). Feasibility
requires \(\sqrt A\ge1+\sqrt c\ge1/\rho^2\), and squaring gives
\eqref{eq:rhomin} after multiplication by \(\rho^6\).

For part (iii), \(c=n\Lambda/\zeta_4=(1+\eta)((1-\rho^2)/\rho^2)^2\), and
\eqref{eq:attained} is \eqref{eq:prd} for this value. When \(\eta=0\), the
right hand side of \eqref{eq:attained} equals \(1/\rho^2\), so the
calculation of part (ii) becomes an equivalence.
\end{proof}

\subsection{Two asset results}\label{app:two}

Throughout this subsection, \(\Sigma\) has unit diagonal and correlation
\(r=\cos\delta\) with \(\delta\in(0,\pi)\), \(L(\tau)\) is the factor in
\eqref{eq:two-Lt}, and we write \(p_1=\cos^2(\tau+\delta)\),
\(\bar p_1=\sin^2(\tau+\delta)\), \(p_2=\cos^2\tau\), and
\(\bar p_2=\sin^2\tau\) for its squared entries, together with
\(\theta=2\tau+\delta\) and \(x=\cos\theta\).

\begin{lemma}[Scaling]\label{lem:two-scale}
Let \(W\) be a positive diagonal matrix. The map \(L\mapsto WL\) is a
bijection between the factors of \(\Sigma\) and those of \(W\Sigma W\), and
replacing the targets \((m_3,m_4)\) by \((W^3m_3,W^4m_4)\) leaves \(u\),
\(w\), and \(d\) unchanged.
\end{lemma}

\begin{proof}
The identity \((WL)(WL)^\top=W\Sigma W\) and the inverse map
\(L'\mapsto W^{-1}L'\) prove the bijection. Since
\((WL)\had k=W^kL\had k\) for \(k=3,4\), the factor \(W^k\) cancels from both
sides of the systems that define \(u\) and \(w\).
\end{proof}

\begin{proposition}[Determinants and closed forms]\label{prop:two-skew-angle}
For every \(\tau\),
\begin{equation}\label{eq:two-sumdiff}
  p_1+p_2=1+rx,\qquad
  p_1-p_2=-\sin\delta\sin\theta,\qquad
  p_1\bar p_2+p_2\bar p_1=\tfrac12(2-r^2-x^2),
\end{equation}
and the entrywise power determinants are
\begin{align}
  D_3&=\det L(\tau)\had3
  =-\frac14\sin\delta\left(\sin^2\delta+3\sin^2\theta\right),\notag\\
  D_4&=\det L(\tau)\had4
  =-\frac12\sin\delta\sin\theta\left(2-r^2-x^2\right).\label{eq:two-skew-D4}
\end{align}
Hence \(L(\tau)\had3\) is always nonsingular, while \(L(\tau)\had4\) is
singular precisely when \(\sin\theta=0\), where
\(p_1=p_2=\frac12(1\pm r)\). For \(S,K\in\R^2\),
\begin{equation}\label{eq:two-skew-u}
  u_1=\frac{\sin^3\tau\,S_1-\sin^3(\tau+\delta)\,S_2}{D_3},\qquad
  u_2=\frac{\cos^3(\tau+\delta)\,S_2-\cos^3\tau\,S_1}{D_3},
\end{equation}
and, when \(\sin\theta\ne0\),
\begin{equation}\label{eq:two-skew-w}
  w_1=\frac{\bar p_2^2K_1-\bar p_1^2K_2}{D_4},\qquad
  w_2=\frac{p_1^2K_2-p_2^2K_1}{D_4}.
\end{equation}
If \(K_1=K_2=K\), then
\begin{equation}\label{eq:two-skew-weq}
  w_1=\frac{2K(1-rx)}{2-r^2-x^2},\qquad
  w_2=\frac{2K(1+rx)}{2-r^2-x^2}.
\end{equation}
\end{proposition}

\begin{proof}
The identities \eqref{eq:two-sumdiff} follow from the product to sum
formulas. Let \(z_1=\cos(\tau+\delta)\sin\tau\) and
\(z_2=\sin(\tau+\delta)\cos\tau\), so that \(z_1-z_2=-\sin\delta\),
\(z_1+z_2=\sin\theta\), and
\(D_3=z_1^3-z_2^3=\frac14(z_1-z_2)\left(3(z_1+z_2)^2+(z_1-z_2)^2\right)\),
which is the formula for \(D_3\). Next,
\(D_4=p_1^2\bar p_2^2-\bar p_1^2p_2^2=(p_1-p_2)(p_1\bar p_2+p_2\bar p_1)\),
and \eqref{eq:two-sumdiff} gives \eqref{eq:two-skew-D4}. Since \(\sin\delta>0\) and
\(2-r^2-x^2\ge1-r^2>0\), the singularity claims follow, and at
\(\sin\theta=0\) the first two identities in \eqref{eq:two-sumdiff} give
\(p_1=p_2=\frac12(1\pm r)\). Cramer's rule gives \eqref{eq:two-skew-u} and
\eqref{eq:two-skew-w}. When \(K_1=K_2=K\), the numerators in
\eqref{eq:two-skew-w} are \(K(\bar p_2-\bar p_1)(\bar p_2+\bar p_1)\) and
\(K(p_1-p_2)(p_1+p_2)\), which \eqref{eq:two-sumdiff} turns into
\(-K\sin\delta\sin\theta(1-rx)\) and \(-K\sin\delta\sin\theta(1+rx)\), and
division by \(D_4\) gives \eqref{eq:two-skew-weq}.
\end{proof}

\begin{proof}[Proof of Lemma~\ref{lem:two-angle}]
The rows of any factor are unit vectors with inner product \(r\), so their
angles differ by \(\delta\) up to a sign, and a sign change of the second
column reverses both angles and hence this sign, after which the factor is
\(L(\tau)\) for some \(\tau\). A simultaneous sign change of both columns
replaces \(\tau\) by \(\tau+\pi\), so \([0,\pi)\) covers the whole class. A
permutation \(P\) and signs \(\varepsilon\in\{-1,1\}^2\) act by
\(L\mapsto LP\operatorname{diag}(\varepsilon)\), which maps
\(u\mapsto\operatorname{diag}(\varepsilon)P^\top u\), \(w\mapsto P^\top w\),
and \(d\mapsto P^\top d\), so nonsingularity and condition~\eqref{eq:feasible}
are unaffected. Proposition~\ref{prop:two-skew-angle} gives the remaining
claims, because for \(\tau\in[0,\pi)\) the equation \(\sin(2\tau+\delta)=0\)
holds exactly at \(\tau_0\) and \(\tau_0+\pi/2\), where \(L(\tau_0)\) has
columns \(\sin(\delta/2)(-1,1)^\top\) and \(\cos(\delta/2)(1,1)^\top\), the
eigenvectors of \(\Sigma\) scaled by \(\sqrt{1-r}\) and \(\sqrt{1+r}\),
while \(L(\tau_0+\pi/2)\) has the same columns in reverse order up to sign.
\end{proof}

\begin{lemma}[Two component budget]\label{lem:two-skew-pair}
For \(d\in\R^2\), the following statements are equivalent.
\begin{enumerate}
\item[(i)] \(d_1>0\), \(d_2>0\), and \(1/d_1+1/d_2\le1\).
\item[(ii)] \(d_1>1\), \(d_2>1\), and \((d_1-1)(d_2-1)\ge1\).
\end{enumerate}
They imply \(d_1+d_2\ge4\), with equality only when \(d_1=d_2=2\).
\end{lemma}

\begin{proof}
Under (i), \(1/d_i<1\) gives \(d_i>1\), and for positive \(d_i\) the
inequality \(1/d_1+1/d_2\le1\) reads \(d_1d_2-d_1-d_2+1\ge1\), so (i) and
(ii) are equivalent. Under (ii), the arithmetic geometric mean inequality
gives \(\frac12(d_1+d_2)-1\ge\sqrt{(d_1-1)(d_2-1)}\ge1\), with equality
throughout only when \(d_1-1=d_2-1=1\).
\end{proof}

\begin{proof}[Proof of Theorem~\ref{thm:two-sym}]
If a feasible factor exists, a matching distribution exists and each
marginal has unit variance, so \(K_j\ge(\E[(X_j-\mu_j)^2])^2=1\). If
\(K_1=K_2=1\), then \((X_j-\mu_j)^2=1\) almost surely for both coordinates,
so the centre has probability zero and within every component
\(T_i^2L_{1i}^2=T_i^2L_{2i}^2=1\) almost surely, which gives
\(L_{1i}^2=L_{2i}^2\) for both columns and makes the rows of \(L\had4\)
identical, contrary to nonsingularity. This proves necessity.

For sufficiency, \(S=0\) gives \(u=0\) and \(d=w\), and when
\(L(\tau)\had4\) is nonsingular \(w\) is the unique solution of
\begin{equation}\label{eq:two-rows}
  p_j^2w_1+\bar p_j^2w_2=K_j,\qquad j=1,2.
\end{equation}
We use the identity
\begin{equation}\label{eq:two-identity}
  \frac{a^2}{p}+\frac{(1-a)^2}{1-p}=1+\frac{(a-p)^2}{p(1-p)},\qquad
  p\in(0,1),\quad a\in\R,
\end{equation}
which follows by clearing denominators, together with
\(p_1=p_2=\frac12(1-r)\) and \(\bar p_1=\bar p_2=\frac12(1+r)\) at
\(\tau_0=(\pi-\delta)/2\). We also note that \(p_1=p_2\) together with
\(p_1(1-p_1)=0\) would force \(|r|=1\), because \(p_1=1\) gives
\(p_2=r^2\) and \(p_1=0\) gives \(\cos\tau=\cos(\tau+\delta)=0\).

Assume \(K_1<K_2\), let \(\kappa=K_2/K_1-1>0\), and define
\(\Phi(\tau)=p_1-p_2-\sqrt{\kappa p_1\bar p_1}\). At \(\tau_A=\pi-\delta\)
we have \(p_1=1\), \(\bar p_1=0\), and \(p_2=r^2\), so
\(\Phi(\tau_A)=1-r^2>0\), whereas
\(\Phi(\tau_0)=-\frac12\sqrt{\kappa(1-r^2)}<0\). By continuity some \(\tau\)
between \(\tau_0\) and \(\tau_A\) satisfies \(p_1-p_2=\sqrt{\kappa p_1\bar p_1}\),
and at this angle \(p_1\bar p_1>0\), since otherwise \(p_1=p_2\) and
\(|r|=1\), so \(p_1>p_2\) and \(L(\tau)\had4\) is nonsingular. Set
\(w_1=K_1/p_1\) and \(w_2=K_1/\bar p_1\). The first equation in
\eqref{eq:two-rows} holds because \(p_1+\bar p_1=1\), and the second holds
because \eqref{eq:two-identity} with \(a=p_2\) and \(p=p_1\) together with
this root equation give
\(K_1(p_2^2/p_1+\bar p_2^2/\bar p_1)=K_1(1+\kappa)=K_2\). Hence \(d=w\) is
the vector in part (a), \(\sum_i1/d_i=(p_1+\bar p_1)/K_1=1/K_1\le1\), and
\(L(\tau)\) is feasible by Theorem~\ref{thm:blocks}. The case \(K_1>K_2\)
follows by relabelling the two assets, which exchanges the rows of \(L\) and
leaves \(\Sigma\) unchanged.

Finally, let \(K_1=K_2=K>1\), fix \(\kappa\in(0,\min\{1-r^2,K-1\}]\), and
define \(\Psi(\tau)=(p_1-p_2)^2-\kappa(p_1+p_2)(2-p_1-p_2)\). At
\(\tau_*=\pi/4-\delta/2\) we have \(\theta=\pi/2\), so \(p_1+p_2=1\) and
\(p_1-p_2=-\sin\delta\) by \eqref{eq:two-sumdiff}, giving
\(\Psi(\tau_*)=1-r^2-\kappa\ge0\), whereas
\(\Psi(\tau_0)=-\kappa(1-r^2)<0\). By continuity some \(\tau\) between
\(\tau_0\) and \(\tau_*\) satisfies \(\Psi(\tau)=0\), and there
\(p_1\ne p_2\), since \(p_1=p_2\) would force \(p_1(1-p_1)=0\) and hence
\(|r|=1\), so \(L(\tau)\had4\) is nonsingular. Write \(e=\frac12(p_1+p_2)\)
and \(\varrho=\frac12(p_1-p_2)\ne0\), so that \(\Psi(\tau)=0\) reads
\(\varrho^2=\kappa e(1-e)\) and gives \(e\in(0,1)\), and set
\(\nu=K/(1+\kappa)\), \(w_1=\nu/e\), and \(w_2=\nu/(1-e)\). Since
\(p_j-e=\pm\varrho\), identity \eqref{eq:two-identity} with \(a=p_j\) and
\(p=e\) gives
\(p_j^2w_1+\bar p_j^2w_2=\nu(1+\varrho^2/(e(1-e)))=\nu(1+\kappa)=K\) for
both rows, so \(d=w\), \(d_i\ge\nu\), and
\(\sum_i1/d_i=(e+1-e)/\nu=(1+\kappa)/K\le1\).
\end{proof}

\begin{remark}[The limiting case \texorpdfstring{\(K=(1,1)\)}{K=(1,1)}]\label{rem:two-11}
This target is attainable by a column component distribution, although
every matching factor has singular \(L\had4\). With the columns
\(L_1=\ell_-(-1,1)^\top\) and \(L_2=\ell_+(1,1)^\top\) of \(L(\tau_0)\),
where \(\ell_\pm=\sqrt{(1\pm r)/2}\), assign probability \(\ell_-^2\) to the
first component, \(\ell_+^2\) to the second, and nothing to the centre, and
within each component take the displacements \(\pm1/\ell_-\) or
\(\pm1/\ell_+\) with equal conditional probabilities. The support is then
\(\pm(-1,1)\) with total probability \((1-r)/2\) and \(\pm(1,1)\) with total
probability \((1+r)/2\), so both marginals are symmetric two point laws on
\(\{-1,1\}\) and the correlation is \(r\).
\end{remark}

\begin{proof}[Proof of Proposition~\ref{prop:two-symroot}]
The symmetric root has diagonal entries \(\chi_1\) and off diagonal
entries \(\chi_2\), where \(\chi_1=\frac12(\sqrt{1+r}+\sqrt{1-r})\) and
\(\chi_2=\frac12(\sqrt{1+r}-\sqrt{1-r})\), so that \(\chi_1^2+\chi_2^2=1\), \(2\chi_1\chi_2=r\), and
\(\chi_1>|\chi_2|\), which makes both entrywise powers nonsingular. When \(K_1=K_2=K\), the system
\((\Sigma^{1/2})\had4w=K\mathbf1\) is solved by the vector with equal
entries \(K/(\chi_1^4+\chi_2^4)\), and
\(\chi_1^4+\chi_2^4=1-2\chi_1^2\chi_2^2=1-r^2/2\) gives
\(w_1=w_2=2K/(2-r^2)\). If \(S=0\), condition~\eqref{eq:feasible} reads
\(2(2-r^2)/(2K)\le1\), which is \(K\ge2-r^2\).

For general \(S\), Cramer's rule gives
\(u_1=(\chi_1^3S_1-\chi_2^3S_2)/(\chi_1^6-\chi_2^6)\) and
\(u_2=(\chi_1^3S_2-\chi_2^3S_1)/(\chi_1^6-\chi_2^6)\), so
\[
  |u_i|
  \le\frac{\|S\|_\infty(\chi_1^3+|\chi_2|^3)}{\chi_1^6-\chi_2^6}
  =\frac{\|S\|_\infty}{\chi_1^3-|\chi_2|^3}
  =\frac{\|S\|_\infty}{\sqrt{1-|r|}\,(1+|r|/2)},
\]
where the last step is a direct computation in which \(|\chi_2|\) is
obtained from \(\chi_2\) by replacing \(r\) with \(|r|\). The displayed
sufficient condition therefore gives \(d_i=w_i-u_i^2\ge2\) for both
components, and \eqref{eq:feasible} follows.
\end{proof}

\begin{proposition}[Independent assets with equal kurtosis]\label{prop:two-skew-rzero}
Let \(r=0\), \(K_1=K_2=K\), and \(S\in\R^2\). Write \(y=\cos^2\tau\sin^2\tau\), \(U=1-2y\), and \(V=1-3y\). Then
\begin{gather*}
  u=\frac1V\begin{pmatrix}\cos^3\tau\,S_2-\sin^3\tau\,S_1\\
  \cos^3\tau\,S_1+\sin^3\tau\,S_2\end{pmatrix},\qquad
  \|u\|_2^2=\frac{\|S\|_2^2}{V},\\
  w=\frac KU\begin{pmatrix}1\\1\end{pmatrix}\quad(y\ne1/4),
\end{gather*}
and consequently
\begin{equation}\label{eq:two-skew-sumdiff}
  d_1+d_2=\frac{2K}{U}-\frac{\|S\|_2^2}{V}.
\end{equation}
\end{proposition}

\begin{proof}
For \(r=0\) we have \(\delta=\pi/2\), \(\cos(\tau+\delta)=-\sin\tau\),
\(\sin(\tau+\delta)=\cos\tau\), \(\sin^2\theta=\cos^22\tau=1-4y\), hence
\(D_3=-(1-3y)=-V\), and \(x^2=4y\), so \eqref{eq:two-skew-u} and
\eqref{eq:two-skew-weq} give \(u\) and \(w\), while the cross terms cancel in the sum of the squared entries of
\(Vu\), which therefore equals
\((\cos^6\tau+\sin^6\tau)\|S\|_2^2=V\|S\|_2^2\) and gives the norm
identity. Summing \(d_i=w_i-u_i^2\) gives
\eqref{eq:two-skew-sumdiff}, and \(y=1/4\) is the singular case
\(\sin\theta=0\) of Proposition~\ref{prop:two-skew-angle}.
\end{proof}

\begin{proof}[Proof of Proposition~\ref{prop:two-fail}]
By \eqref{eq:two-skew-sumdiff}, \(d_1+d_2=5/(1-2y)-2/(1-3y)\) at every
nonsingular angle, and \(d_1+d_2<4\) is equivalent, after multiplication by
the positive denominators, to \(24y^2-9y+1>0\), which holds for every \(y\)
because the discriminant is negative. Lemma~\ref{lem:two-skew-pair}
therefore rules out feasibility. For attainability,
Theorem~\ref{thm:blocks} in one dimension with \(L=1\), \(u=1\),
\(w=5/2\), and \(d=3/2\) gives \(a=\frac12(1+\sqrt7)\),
\(b=\frac12(\sqrt7-1)\), and the three point law displayed in the
proposition, and two independent copies have covariance \(I\) together with
the required marginal moments.
\end{proof}

\begin{proof}[Proof of Theorem~\ref{thm:two-skew-pca}]
At \(\tau_0\) we have \(\cos\tau_0=\sin(\delta/2)\),
\(\sin\tau_0=\cos(\delta/2)\), \(\cos(\tau_0+\delta)=-\sin(\delta/2)\),
\(\sin(\tau_0+\delta)=\cos(\delta/2)\), and \(D_3=-\frac14\sin^3\delta\),
so \eqref{eq:two-skew-u} gives \(u_1(\tau_0)=-(S_1-S_2)/(2\sin^3(\delta/2))\)
and \(u_2(\tau_0)=(S_1+S_2)/(2\cos^3(\delta/2))\), and \(2\sin^2(\delta/2)=1-r\) and \(2\cos^2(\delta/2)=1+r\) give
\(u_1(\tau_0)^2=2\gamma_-^2/(1-r)\) and \(u_2(\tau_0)^2=2\gamma_+^2/(1+r)\).
Because the kurtoses are equal, \(w(\tau)\) is given by
\eqref{eq:two-skew-weq}, whose denominator is at least \(1-r^2\), so \(w\)
is continuous at \(\tau_0\) with value \((2K/(1-r),2K/(1+r))\) at \(x=-1\),
and consequently \(d(\tau)\) converges, as \(\tau\to\tau_0\), to
\(d^0=(2(K-\gamma_-^2)/(1-r),\,2(K-\gamma_+^2)/(1+r))\). The assumptions state that \(d^0>0\) and \(1/d_1^0+1/d_2^0<1\). Both strict
inequalities persist by continuity for every \(\tau\ne\tau_0\) in a
neighbourhood of \(\tau_0\), where \(L(\tau)\had4\) is nonsingular, and
Theorem~\ref{thm:blocks} gives feasibility.
\end{proof}

\bibliographystyle{elsarticle-harv}
\bibliography{references}

\end{document}

%% file: tab_rotation_gain.tex
\begin{tabular}{lrcccc}
\toprule
& & \multicolumn{2}{c}{GenUT} & \multicolumn{2}{c}{PRD}\\
\cmidrule(lr){3-4}\cmidrule(lr){5-6}
Data & \(n\) & Standard & Rotated & Standard & Rotated\\
\midrule
FTSE & 20 & 0/1 & 1/1 & 1/1 & 1/1 \\
IPSA 2010 & 35 & 0/1 & 1/1 & 0/1 & 1/1 \\
IPSA weekly & 33 & 0/33 & 33/33 & 10/33 & 32/33 \\
IPSA daily & 31 & 0/28 & 28/28 & 15/28 & 28/28 \\
\addlinespace
Synthetic & 5 & 10/30 & 30/30 & 30/30 & 30/30 \\
Synthetic & 10 & 0/30 & 30/30 & 9/30 & 30/30 \\
Synthetic & 20 & 0/30 & 30/30 & 0/30 & 30/30 \\
Synthetic & 40 & 0/30 & 30/30 & 0/30 & 30/30 \\
Synthetic & 80 & 0/12 & 12/12 & 0/12 & 12/12 \\
Synthetic & 160 & 0/12 & 12/12 & 0/12 & 12/12 \\
\midrule
All & & 10/207 & 207/207 & 65/207 & 206/207 \\
\bottomrule
\end{tabular}